\documentclass[10pt,leqno]{amsart}
   \usepackage[centertags]{amsmath}
   \usepackage{amsfonts}
   \usepackage{amsmath}
   \usepackage{amssymb}
   \usepackage{amsthm}
   \usepackage{newlfont}
   \usepackage[latin1]{inputenc}
\usepackage{multirow, bigdelim}

\usepackage{graphicx}

\newcommand{\dosfilas}[2]{
  \ldelim[{2}{2mm}& #1 &\rdelim]{2}{2mm} \\
  & #2 & &  & &
}

\allowdisplaybreaks[3]

\theoremstyle{plain}
\newtheorem{theorem}{Theorem}[section]
\newtheorem{lemma}[theorem]{Lemma}

\newtheorem{corollary}[theorem]{Corollary}

\theoremstyle{definition}

\theoremstyle{remark}
\newtheorem{remark}[theorem]{Remark}
\newtheorem*{remark*}{Remark}

\numberwithin{equation}{section}

\newcommand\D{{\mathcal D}}
\newcommand\A{{\mathcal A}}

\newcommand\F{{\mathcal F}}
\newcommand\I{{\mathcal I}}
\newcommand\U{{\mathcal{U}}}
\newcommand\h{{\mathfrak h}}

\newcommand\CC{{\mathbb C}}

\newcommand\ZZ{{\mathbb Z}}

\newcommand\PP{{\mathbb P}}
\newcommand\II{{\mathbb I}}
\newcommand\UU{{\mathbb U}}

\newcommand{\Cas}{\mathrm {Cas\,}}
\newcommand\p{\mbox{$\mathfrak{p}$}}

\newcommand\Sh{\mbox{\Large $\mathfrak {s}$}}

\title[]
{Constructing Krall-Hahn orthogonal polynomials}

\author{Antonio J. Dur\'an and Manuel D. de la Iglesia}
\address{A. J. Dur\'an\\
Departamento de An\'{a}lisis Matem\'{a}tico \\
Universidad de Sevilla \\
Apdo (P. O. BOX) 1160\\
41080 Sevilla. Spain.}
\email{duran@us.es}
\address{Manuel D. de la Iglesia\\
Instituto de Matem\'aticas \\
Universidad Nacional Aut\'onoma de M\'exico \\
Circuito Exterior, C.U.\\
04510, Mexico D.F. Mexico.}
\email{mdi29@im.unam.mx}
\thanks{Partially supported by MTM2012-36732-C03-03 (Ministerio de Economía y Competitividad),
FQM-262, FQM-4643, FQM-7276 (Junta de Andalucía) and Feder Funds (European
Union).}



   \date{}
   
\begin{document}
   \maketitle

   \begin{abstract}
Given a sequence of polynomials $(p_n)_n$, an algebra of operators $\A$ acting in the linear space of polynomials and an operator $D_p\in \A$
with $D_p(p_n)=\theta_np_n$, where $\theta_n$ is any arbitrary eigenvalue, we construct a new sequence of polynomials $(q_n)_n$ by considering a linear combination of $m+1$ consecutive $p_n$: $q_n=p_n+\sum_{j=1}^m\beta_{n,j}p_{n-j}$. Using the concept of $\mathcal{D}$-operator, we determine the structure of the sequences $\beta_{n,j}, j=1,\ldots,m,$ in order that the polynomials $(q_n)_n$ are eigenfunctions of an operator in the algebra $\A$. As an application, from the classical discrete family of Hahn polynomials we construct orthogonal polynomials $(q_n)_n$
which are also eigenfunctions of higher-order difference operators.
\end{abstract}

\section{Introduction}
The  issue of orthogonal polynomials which are also eigenfunctions of a higher-order differential operator was raised by H. L. Krall in 1939, when he obtained a complete classification for the case of a differential operator of order four (\cite{Kr2}). After his pioneer work, orthogonal polynomials which are also eigenfunctions of higher-order differential operators are usually called Krall polynomials. This terminology can be extended for finite order difference and $q$-difference operators. Krall polynomials are also called bispectral, following the terminology introduced by Duistermaat and Gr\"unbaum (\cite{DG}; see also \cite{GrH1}, \cite{GrH3}).

Regarding Krall polynomials, there are important differences depending whether one considers differential or difference operators.
Indeed, roughly speaking, one can construct Krall polynomials $q_n(x)$, $n\ge 0$,  by using the Laguerre $x^\alpha e^{-x}$, or Jacobi weights $(1-x)^\alpha(1+x)^\beta$, assuming that one or two of the parameters $\alpha$ and $\beta$ are nonnegative integers and adding a linear combination of Dirac deltas and their derivatives at the endpoints of the orthogonality interval (\cite{Kr2}, \cite{koekoe}, \cite{koe}, \cite{koekoe2}, \cite{L1}, \cite{L2}, \cite{GrH1}, \cite{GrHH}, \cite{GrY}, \cite{Plamen1}, \cite{Plamen2}, \cite{Zh}). This procedure of adding deltas seems not to work if we want to construct Krall discrete polynomials from the classical discrete measures of Charlier, Meixner, Krawtchouk and Hahn (see the papers \cite{BH} and \cite{BK} by Bavinck, van Haeringen and Koekoek answering, in the negative,
a question posed by R. Askey in 1991 (see pag. 418 of \cite{BGR})).

As it has been shown recently by one of us, instead of adding deltas, Krall discrete polynomials can be constructed by multiplying the classical discrete weights by certain polynomials (see \cite{du0}).
The kind of transformation which consists in multiplying a measure $\mu$ by a polynomial $r$ is called a Christoffel transform.
It has a long tradition in the context of orthogonal polynomials: it goes back a century and a half ago when
E. B. Christoffel (see \cite{Chr} and also \cite{Sz}) studied it for the particular case $r(x)=x$.

Using suitable polynomials, a number of conjectures have been posed in \cite{du0} on how to construct bispectral polynomials from
the families of Charlier, Meixner, Krawtchouk and Hahn. For Charlier, Meixner and Krawtchouk, those conjectures have been proved by the authors in \cite{ddI}. Those families have in common that the eigenvalues for their associated second-order difference operator are linear sequences in $n$.

The purpose of this paper is to prove the conjecture 5 in \cite{du0} for the Hahn polynomials. In this case, the eigenvalues for its associated second-order difference operator are now a second degree polynomial in $n$. We consider higher-order difference operators of the form
\begin{equation}\label{doho}
D=\sum_{l=s}^rh_l\Sh _l, \quad s\le r,\quad s,r\in \ZZ,
\end{equation}
where $h_l$ are polynomials and $\Sh_l$ stands for the shift operator $\Sh_l(p)=p(x+l)$. If $h_r,h_s\not =0$, the order of $D$ is then $r-s$. We also say that $D$ has genre $(s,r)$.

The conjecture we will prove here is the following:

\medskip
\noindent
\textbf{Conjecture.}
Let $\rho_{a,b,N} $ be the Hahn weight (see Section \ref{SEC4} for details). Given a quartet of finite sets $\F=(F_1,F_2,F_3,F_4)$ of positive integers (the empty set is allowed) consider the weight $\rho _{a,b,N}^{\F}$ defined by
\begin{equation}\label{udspmi}
\rho _{a,b,N}^{\F}=\prod_{f\in F_1}(b+N+1+f-x)\prod_{f\in F_2}(x+a+1+f)\prod_{f\in F_3}(N-f-x)\prod_{f\in F_4}(x-f)\rho _{a,b,N}.
\end{equation}
Assume that the measure $\rho _{a,b,N}^{\F}$ has an associated sequence of orthogonal polynomials. Then they are eigenfunctions of a higher-order difference operator of the form (\ref{doho}) with
$$
-s=r=\sum_{f\in F_1,F_2,F_3,F_4}f-\sum_{i=1}^4\binom{n_{F_i}}{2}+1,
$$
where $n_F$ denotes the cardinal of $F$.
\medskip

The content of this paper is as follows. In order to prove the conjecture above, we use the approach developed in \cite{ddI} for constructing Krall polynomials from the Charlier, Meixner and Krawtchouk families. The main ingredients of this approach will be considered in Sections \ref{SEC2} and \ref{SEC3}. The
first ingredient is the $\D$-operators (see Section 3). This is an abstract concept introduced in \cite{du1} by one of us which has shown to be very useful to generate Krall, Krall discrete and $q$-Krall families of polynomials (see \cite{du1}, \cite{AD}, \cite{ddI}, \cite{ddI1}, \cite{dudh}).
For a positive number $m$ and $m$ polynomials $Y_j$, $j=1,\ldots, m$, (which act as parameters)
we can construct from the classical discrete families $(p_n)_n$, using $\D$-operators, a huge class of families of polynomials $(q_n)_n$ which are also eigenfunctions of difference operators of the form (\ref{doho}). The sequence of polynomials $(q_n)_n$ are not in general orthogonal.
The second ingredient establishes how to choose the polynomials $Y_j$'s such that the polynomials $(q_n)_n$ are also orthogonal with respect to a measure. As for the case of Charlier, Meixner and Krawtchouk (studied in \cite{ddI}), this second ingredient turns into a very nice symmetry between the Hahn family and the polynomials $Y_j$'s. Indeed, the polynomials $Y_j$'s can be chosen to be dual Hahn polynomials, but with a suitable modification of the parameters.

In Sections \ref{SEC4} and \ref{sch} we will put together all these ingredients to construct bispectral Krall-Hahn orthogonal polynomials and prove the conjecture above.

\section{Preliminaries}\label{PREM}

For a linear operator $D$ acting in the linear space of polynomials $\mathbb{P}$, i.e. $D:\mathbb{P}\rightarrow\mathbb{P}$, and a polynomial $P(x)=\sum _{j=0}^ka_jx^j$, the operator $P(D)$ is defined in the usual way as $P(D)=\sum _{j=0}^ka_jD^j$.

Let $\mu$ be a moment functional on the real line, that is, a linear mapping $\mu:\mathbb{P}\rightarrow\mathbb{R}$. It is well-known that any moment
functional on the real line can be represented by integrating with respect to a Borel measure (positive or not) on the real line (this representation is not unique, see \cite{du-1}). If we also denote this measure by $\mu$, we have $\langle \mu, p\rangle=\int p(x)d\mu(x)$ for all polynomials $p\in\mathbb{P}$. Taking this into account, we will conveniently use along this paper one
or other terminology (orthogonality with respect to a moment functional or with
respect to a measure). We say that a sequence of polynomials $(p_n)_n$, $p_n$ of degree $n$, $n\geq0$, is orthogonal with respect to the moment functional $\mu$ if $\langle \mu, p_np_m\rangle=0$, for $n\not =m$ and $\langle \mu, p_n^2\rangle\not =0$.  Since the Hahn polynomials considered in this paper are orthogonal with respect to a degenerate measure (the measure is a finite combination of Dirac deltas), we will stress this property of non-vanishing norms when necessary.

\medskip

As we wrote in the Introduction, the kind of transformation which consists in multiplying a moment functional $\mu$ by a polynomial $r$ is called a Christoffel transform. The new moment functional $r\mu$ is defined by $\langle r\mu,p\rangle =\langle \mu,rp\rangle $. Its reciprocal is the Geronimus transform $\tilde \mu$ which satisfies $r\tilde \mu =\mu$. Notice that the Geronimus transform of the moment functional $\mu$ is not uniquely defined. Indeed, write $a_i$, $i=1,\ldots , u$, for the different real roots of the polynomial $r$, each one with multiplicity $b_i$, respectively.
It is easy to see that if $\tilde \mu$ is a Geronimus transform of $\mu$ then the moment functional $\tilde \mu +\sum_{i=1}^u\sum_{j=0}^{b_i-1}\alpha_{i,j}\delta _{a_i}^{(j)}$ is also a Geronimus transform of $\mu$, where $\alpha_{i,j}$ are real numbers. These numbers are usually called the free parameters of the Geronimus transform.

In the literature, Geronimus transform is sometimes called Darboux transform with parameters while Christoffel transform is called Darboux transform without parameters. The reason is the following. The three-term recurrence relation
\begin{equation*}\label{fvo}
x p_n(x )=a_{n+1}p_{n+1}(x)+b_np_n(x)+c_np_{n-1}(x), \quad n\geq0,
\end{equation*}
for the orthogonal polynomials $(p_n)_n$ with respect to $\mu$ can be rewritten as $x p_n=J(p_n)$, where $J$ is the second-order difference operator $J=a_{n+1}\Sh_1+b_n\Sh_0+c_n\Sh_{-1}$ and $\Sh_l$ is the shift operator (acting on the discrete variable $n$): $\Sh_l(x_n)=x_{n+l}$.
For any $\lambda \in \CC $, decompose $J$ into $J=AB+\lambda I$
whenever it is possible, where $A=\alpha_n\Sh_0+\beta_n\Sh_1$ and $B= \delta_n\Sh_{-1}+\gamma_n\Sh_0$. We then call $\tilde J=BA+\lambda I$
a Darboux transform of $J$ with parameter $\lambda$. It turns out that the second-order difference operator $\tilde J$ associated to a Geronimus transform $\tilde \mu$ of $\mu$ can be obtained by applying a sequence of  $k$ Darboux transforms (with parameters $\lambda_i$, $i=1,\ldots, k$) to the operator $J$ associated to the measure $\mu$. This kind of Darboux transform has been used by Gr\"unbaum, Haine, Hozorov, Yakimov and Iliev to construct Krall and $q$-Krall polynomials (\cite{GrHH}, \cite{GrY}, \cite{HP}, \cite{Plamen1} or \cite{Plamen2}).

The family of measures $\rho _{a,b,N}^\F$ \eqref{udspmi} in the Introduction is defined by applying a Christoffel transform to the Hahn weight. But, it turns out that they can also be defined by using the Geronimus transform. This Geronimus transform is however defined by a different polynomial. We also have to make
a suitable choice of the free parameters of this Geronimus transform and apply it to a Hahn weight but maybe with different parameters and affected by a shift in the variable. The following example will clarify this point. Consider $F_1=\{1\}$, $F_2=\{ 1\}$, $F_3=\{1\}$, $F_4=\{1\}$ and the Christoffel transform $\rho_{a,b,N}^{\F}$ of the Hahn weight $\rho _{a,b,N}$ defined by the polynomial $p(x)=(b+N+2-x)(x+a+2)(N-1-x)(x-1)$. That is,
$$
\rho_{a,b,N}^{\F}=(b+N+2-x)(x+a+2)(N-1-x)(x-1)\rho _{a,b,N}.
$$
From the definition of the Hahn weight $\rho _{a,b,N}$ (see \eqref{pesoH} below), we have after a simple computation
\begin{align*}
\rho_{a,b,N}^{\F}=&-\frac{\Gamma(N+b+3)(N-1)(a+2)}{(N+b+1)\Gamma(b+1)}\delta_0-\frac{\Gamma(N+a+3)(N-1)(b+2)}{(N+a+1)\Gamma(a+1)}\delta_N\\
&\frac{\Gamma(N+1)}{\Gamma(a+1)\Gamma(b+1)}\sum_{x=2}^{N-2}\frac{\Gamma(N-x+b+3)\Gamma(x+a+3)}{(a+x+1)(N-x+b+1)x(x-2)!(N-x-2)!}\delta_x
\end{align*}
This shows that
$$
x(N-x)(a+x+1)(N-x+b+1)\rho_{a,b,N}^{\F}=(N-3)_4(a+1)_4(b+1)_4\rho_{a+4,b+4,N-4}(x-2),
$$
where $(z)_j=z(z+1)\cdots(z+j-1)$ stands for the Pochhammer symbol.

That is, $\rho_{a,b,N}^{\F}$ is also the Geronimus transform defined by the polynomial $x(N-x)(a+x+1)(N-x+b+1)$ of the Hahn weight $\rho_{a+4,b+4,N-4}(x-2)$ where the free parameters (associated to the roots $x=0, N, -a-1, N+b+1$) have to be necessarily chosen equal to $-\frac{\Gamma(N+b+3)(N-1)(a+2)}{(N+b+1)\Gamma(b+1)}$, $\frac{\Gamma(N+a+3)(N-1)(b+2)}{(N+a+1)\Gamma(a+1)}$, 0 and $0$, respectively.

\medskip

Along this paper, we use the following notation: given a finite set of positive integers $F=\{f_1,\ldots, f_m\}$, the expression
\begin{equation}\label{defdosf}
  \begin{array}{@{}c@{}lccc@{}c@{}}
    & &&\hspace{-1.3cm}{}_{j=1,\ldots , m} \\
    \dosfilas{ z_{f,j} }{f\in F}
  \end{array}
\end{equation}
inside of a matrix or a determinant will mean the submatrix defined by
$$
\begin{pmatrix}
z_{f_1,1}&z_{f_1,2}&\cdots&z_{f_1,m}\\
\vdots&\vdots&\ddots&\vdots\\
z_{f_m,1}&z_{f_m,2}&\cdots&z_{f_m,m}
\end{pmatrix}
$$

\section{$\D$-operators}\label{SEC2}
The concept of $\D$-operator was introduced by one of us in \cite{du1}. In  \cite{du1}, \cite{ddI}, \cite{ddI1} and \cite{AD}, it has been showed that $\D$-operators turn out to be an extremely useful tool of an unified method to generate families of polynomials which are eigenfunctions of higher-order differential, difference or $q$-difference operators.
Hence, we start by reminding the concept of $\D$-operator.

The starting point is a sequence of polynomials $(p_n)_n$, $\deg p_n=n$, and an algebra of operators $\A $ acting in the linear space of
polynomials $\mathbb{P}$. For the Hahn polynomials, we will consider the algebra $\A$ formed by all finite order difference operators, i.e.
\begin{equation}\label{algdiffd}
\A =\left\{ \sum_{l=s}^rh_l\Sh_l : h_l\in \PP, l=s,\ldots,r, s\leq r \right\},
\end{equation}
where $\Sh_l$ stands for the shift operator $\Sh_l(p)=p(x+l)$. If $h_r,h_s\not =0$, the order of $D$ is then $r-s$. We also say that $D$ has genre $(s,r)$.

In addition, we assume that the polynomials $p_n$, $n\ge 0$, are eigenfunctions of certain operator $D_p\in \A$. We write $(\theta_n)_n$ for the corresponding eigenvalues, so that $D_p(p_n)=\theta_np_n$, $n\ge 0$. Although for the Hahn polynomials $\theta_n$ is a second degree polynomial in $n$, in this section we do not assume any constraint on the sequence $(\theta_n)_n$.

Given two sequences of numbers $(\varepsilon_n)_n$ and $(\sigma_n)_n$, a $\D$-operator associated to the algebra $\A$ and the sequence of polynomials $(p_n)_n$ is defined as follows. We first consider  the operator $\D :\PP \to \PP $ defined by linearity from
\begin{equation*}\label{defTo}
\D (p_n)=-\frac{1}{2}\sigma_{n+1}p_n+\sum _{j=1}^n (-1)^{j+1}\sigma_{n-j+1}\varepsilon _n\cdots \varepsilon _{n-j+1}p_{n-j},\quad n\ge 0.
\end{equation*}
We then say that $\D$ is a $\D$-operator if $\D\in \A$. In \cite{du1} this kind of $\D$-operator was called of type 2; $\D$-operators of type 1 appears when the sequence $(\sigma_n)_n$ is constant.

As it was shown in \cite{du1}, \cite{dudh} and \cite{ddI}, $\D$-operators of type 1 are useful when the eigenvalues $\theta_n$ are linear in $n$. Otherwise, we have to consider $\D$-operators of type 2. The purpose of this section is to extend the method developed in \cite{ddI} to $\D$-operators of type 2. Therefore we will be able to construct, from the sequence $(p_n)_n$,  new sequences of polynomials $(q_n)_n$ such that there exists an operator $D_q\in \A$ for which they are eigenfunctions.

To do that, we will consider a combination of $m$, $m\ge 1$, consecutive $p_n$'s.
We also use $m$ arbitrary polynomials $Y_1, Y_2, \ldots, Y_m,$ and $m$ $\D$-operators, $\D_1, \D_2, \ldots, \D_m,$ (not necessarily different) defined by the pairs of sequences $(\varepsilon_n^h)_n$, $(\sigma_n^h)_n$, $h=1,\ldots , m$:
\begin{equation}\label{Dh}
\mathcal{D}_h(p_n)=-\frac{1}{2}\sigma_{n+1}^hp_n+\sum _{j=1}^n (-1)^{j+1}\sigma_{n-j+1}^h\varepsilon_{n}^{h}\cdots\varepsilon_{n-j+1}^{h}p_{n-j}.
\end{equation}

We will assume that for $h=1,2,\ldots,m$, the sequences $(\varepsilon_{n}^{h})_n$ and $(\sigma_{n}^{h})_n$ are rational functions in $n$. We write $\xi_{x,i}^h$, $i\in\ZZ$ and $h=1,2,\ldots,m$, for the auxiliary functions defined by
\begin{equation}\label{defxi}
\xi_{x,i}^h=\prod_{j=0}^{i-1}\varepsilon_{x-j}^{h}, \quad i\ge 1,\quad \quad \xi_{x,0}^h=1,\quad\quad \xi_{x,i}^h=\frac{1}{\xi_{x-i,-i}^h},\quad i\leq-1.
\end{equation}
We will consider the $m\times m$ (quasi) Casorati determinant defined by
\begin{equation}\label{casd1}
\Omega (x)=\det \left(\xi_{x-j,m-j}^lY_l(\theta_{x-j})\right)_{l,j=1}^m.
\end{equation}

The details of our method are included in the following Theorem:

\begin{theorem}\label{Teor1} Let $\A$ and $(p_n)_n$ be, respectively, an algebra of operators acting in the linear space of polynomials, and a sequence of polynomials $(p_n)_n$, $\deg p_n=n$. We assume that $(p_n)_n$ are eigenfunctions of an operator $D_p\in \A$, that is, there exist numbers $\theta_n, n\geq0$, such that $D_p(p_n)=\theta_np_n$, $n\ge 0$. We also have $m$ pairs of sequences of numbers $(\varepsilon_n^h)_n$, $(\sigma_n^h)_n$, $h=1,\ldots , m$, which define $m$ $\D$-operators $\D_1,\ldots,\D_m$ (not necessarily different) for $(p_n)_n$ and $\A$ (see \eqref{Dh})) and assume that for $h=1,2,\ldots,m$, each one of the sequences $(\varepsilon_{n}^{h})_n$, $(\sigma_n^h)_n$ is a rational function in $n$.

Let $Y_1, Y_2, \ldots, Y_m,$ be $m$ arbitrary polynomials satisfying that $\Omega (n)\not =0$, $n\ge 0$, where $\Omega $ is the Casorati determinant defined by \eqref{casd1}.

Consider the sequence of polynomials $(q_n)_n$ defined by
\begin{equation}\label{qus}
q_n(x)=\begin{vmatrix}
               p_n(x) & -p_{n-1}(x) & \cdots & (-1)^mp_{n-m}(x) \\
               \xi_{n,m}^1Y_1(\theta_{n}) &  \xi_{n-1,m-1}^1Y_1(\theta_{n-1}) & \cdots & Y_1(\theta_{n-m}) \\
               \vdots & \vdots & \ddots & \vdots \\
                \xi_{n,m}^mY_m(\theta_{n}) &  \xi_{n-1,m-1}^mY_m(\theta_{n-1}) & \cdots & Y_m(\theta_{n-m})
             \end{vmatrix}.
\end{equation}
For a rational function $S$, we define the function $\lambda_x$ by
\begin{equation}\label{lamd}
\lambda_x-\lambda_{x-1}=S(x)\Omega(x),
\end{equation}
and for $h=1,\ldots,m$, we define the function $M_h(x)$ by
\begin{equation}\label{emeiexp}
M_h(x)=\sum_{j=1}^m(-1)^{h+j}\xi_{x,m-j}^hS(x+j)\det\left(\xi_{x+j-r,m-r}^{l}Y_l(\theta_{x+j-r})\right)_{l\in \II_h;r\in\II_j},
\end{equation}
where $\II_h=\{1,2,\ldots , m\}\setminus \{ h\}$. Assume the following:
\begin{align}\label{ass0}
&\mbox{$S(x)\Omega (x)$ is a polynomial in $x$.}\\\label{ass1}
&\mbox{There exist $\tilde{M}_1,\ldots,\tilde{M}_m,$  polynomials in $x$ such that}\\\nonumber
&\hspace{1cm}\mbox{$M_h(x)=\sigma_{x+1}^h\tilde{M}_h(\theta_x)$, $h=1,\ldots,m.$}\\\label{ass2}
&\mbox{There exists a polynomial $P_S$ such that $\displaystyle P_S(\theta_x)=2\lambda_x+\sum_{h=1}^mY_h(\theta_x)M_h(x)$.}
\end{align}
Then there exists an operator $D_{q,S}\in \A$ such that
$$
D_{q,S}(q_n)=\lambda_nq_n,\quad n\ge 0.
$$
Moreover, the operator $D_{q,S}$ is defined by
\begin{equation}\label{Dq}
D_{q,S}=\frac{1}{2}P_S(D_p)+\sum_{h=1}^m\tilde M_h(D_p)\D_hY_h(D_p),
\end{equation}
where $D_p\in \A$ is the operator for which the polynomials $(p_n)_n$ are eigenfunctions.
\end{theorem}

\begin{proof}
The definition of \eqref{Dh}, \eqref{lamd} and \eqref{Dq}, along with the hypotheses \eqref{ass0}-\eqref{ass2}, gives
\begin{align*}
D_{q,S}(p_n)=&\left(\frac{1}{2}P_S(\theta_n)-\frac{1}{2}\sum_{h=1}^m\sigma_{n+1}^h\tilde{M}_h(\theta_n)Y_h(\theta_n)\right)p_n\\
&\qquad+\sum_{h=1}^m\sum_{j=1}^n(-1)^{j+1}\tilde{M}_h(\theta_{n-j})\varepsilon_{n}^h\cdots\varepsilon_{n-j+1}^h\sigma_{n-j+1}^hY_h(\theta_n)p_{n-j}\\
=&\lambda_np_n+\lambda_{n,1}p_{n-1}-\lambda_{n,2}p_{n-2}+\cdots+(-1)^{m+1}\lambda_{n,m}p_{n-m}\\ &\hspace{4cm}
+\sum_{j=m+1}^n(-1)^{j+1}\lambda_{n,j}p_{n-j}
\end{align*}
where
\begin{equation*}\label{lambj}
\lambda_{n,j}=\sum_{h=1}^m\varepsilon_{n}^h\cdots\varepsilon_{n-j+1}^hM_h(n-j)Y_h(\theta_n),\quad j=1,2,\ldots,n.
\end{equation*}
Now we just have to follow exactly the same lines as the proof in Section 8 of \cite{ddI} (from the formula (8.4) on).
\end{proof}

\begin{remark}
We now see that the polynomial $P_S$ (\ref{ass2}) also satisfies
\begin{equation}\label{Pdiff}
P_S(\theta_x)-P_S(\theta_{x-1})=S(x)\Omega(x)+S(x+m)\Omega(x+m).
\end{equation}
Indeed, $P_S$ defined in \eqref{ass2} can be written, using (8.4) and (8.6) in the proof of Theorem 3.2 of \cite{ddI}, as
$$
P_S(\theta_x)=2\lambda_x+\sum_{h=1}^mS(x+h)\Omega_{-h+1,h}(x+1),
$$
where $\Omega_{i,j}(x)$ is the $m\times m$ determinant defined in the same way as $\Omega (x-i)$ but replacing the $j$-th column of $\Omega (x-i)$ by the column vector
$$
\left(\xi_{x-1,m+i-1}^1Y_1(\theta_{x-1}),\ldots ,\xi_{x-1,m+i-1}^mY_m(\theta_{x-1})\right)^t.
$$
In particular $\Omega_{0,1}(x)=\Omega(x)$.

Now, using \eqref{lamd} and the relations (which easily follows by definition)
$$
\Omega_{-h+1,h}(x+1)=\Omega_{-h,h+1}(x),\quad h=1,\ldots,m-1,
$$
and
$$
\Omega_{-m+1,m}(x+1)=\Omega_{0,1}(x+m)=\Omega(x+m),
$$
we obtain \eqref{Pdiff}.
\end{remark}

\begin{remark}
For the particular cases of Laguerre, Jacobi or Askey-Wilson polynomials, one can found Casorati determinants similar to (\ref{qus}) in \cite{GrHH}, \cite{GrY}, \cite{HP}, \cite{Plamen1} or \cite{Plamen2}.
\end{remark}

In Sections \ref{SEC4} and \ref{sch},  we will apply Theorem \ref{Teor1} to the Hahn polynomials. We will see there that
the degree of the polynomial $P_S$ (see \eqref{ass2}) gives the order of the difference operator $D_{q,S}$ \eqref{Dq} with respect to which the new polynomials $(q_n)_n$ are eigenfunctions. This will be a consequence of the following Lemma:

\begin{lemma}\label{lgp2} With the same notation as in Theorem \ref{Teor1}, write
$$
\Psi_j^h(x)=\xi_{x-j,m-j}^hS(x)
\det\left(\xi_{x-r,m-r}^{l}Y_l(\theta_{x-r})\right)_{l\in \II_h;r\in \II_j}, \quad h,j=1,\ldots , m,
$$
and $\Omega _g^h$, $h=1,\ldots , m$, $g=0,1,2,\ldots$, for the particular case of $\Omega$ when $Y_h(x)=x^g$.
Assume that $\theta_x$ is a polynomial in $x$ of degree $2$ and that $\Psi_j^h$, $h,j=1,\ldots , m$, are polynomials in $x$ and write $\tilde d=\max\{\deg \Psi_j^h:h,j=1,\ldots , m\}$.
Then $M_h$ and $S\Omega _g^h$, $h=1,\ldots , m$, $g=0,1,2,\ldots$, are also polynomials in $x$. If, in addition, we assume that the degree of $S\Omega _g^h$ is at most $2g+\deg (S\Omega_0^h)$ for $h=1,\ldots, m,$ and $g\le \tilde d-\deg (S\Omega_0^h)$, then $M_h$ is a polynomial of degree at most $\deg (S\Omega_0^h)$.
\end{lemma}

\begin{proof}
It is analogous to the proof of Lemma 3.2 in \cite{dudh}.
\end{proof}

\section{Two more ingredients}\label{SEC3}

In this Section we will assume that the polynomials $(p_n)_n$, $p_n$ of degree $n$, $n\ge 0$, are orthogonal with respect to a moment functional $\rho$. We automatically have that the polynomials $(p_n)_n$ satisfy the three-term recurrence relation ($p_{-1}=0$)
\begin{equation}\label{ttrr}
x p_n(x)=a_{n+1}p_{n+1}(x)+b_np_n(x)+c_np_{n-1}(x), \quad n\ge 0.
\end{equation}
The measure $\rho$ might be degenerate, in which case for some $n_0$ we might have $a_{n_0}c_{n_0}=0$.

The goal of this Section is to show how to choose appropriately the polynomials $Y_1,\ldots ,Y_m,$ in order that the polynomials $(q_n)_n$ (\ref{qus}) are also orthogonal with respect to certain measure. To stress the dependence of the polynomials $(q_n)_n$  on the polynomials $Y_1,\ldots ,Y_m$, we write $q_n=\Cas_n^{Y_1,\ldots ,Y_m}$.

\subsection{When are the polynomials $(\Cas_n^{Y_1,\ldots ,Y_m})_n$ orthogonal?} \label{ssi}
Only for a convenient choice of the polynomials $Y_j$, $j=1,\ldots, m$, the polynomials $(q_n)_n$ (\ref{qus}) are also orthogonal with respect to a measure. In Section 4 of \cite{ddI}, a method to check the orthogonality of the polynomials $(q_n)_n$ (\ref{qus}) was provided.
This tool assumed that the sequences $(\varepsilon_{n}^{h})_{n\in \ZZ }$, $h=1,\ldots ,m$, do not vanish for any $n$. This is not the case for the Hahn polynomials, hence we use the modified version of that tool given in \cite{dudh}
which includes also that case.

Indeed, given the sequences (in the integers) of the three-term recurrence relation $(a_n)_{n\in \ZZ }$, $(b_n)_{n\in \ZZ }$, $(c_n)_{n\in \ZZ }$ (see \eqref{ttrr} ) and
$(\varepsilon_{n}^{h})_{n\in \ZZ }$, $h=1,\ldots ,m$, assume we have for each $j\ge 0$ and $h=1,\ldots ,m$, one more sequence denoted by $(Z_{j}^{h}(n))_{n\in \ZZ }$ satisfying
\begin{equation}\label{relaR}
\varepsilon_{n+1}^{h}a_{n+1}Z_{j}^{h}(n+1)-b_nZ_{j}^{h}(n)+\frac{c_n}{\varepsilon_{n}^{h}}Z_{j}^{h}(n-1)=(\eta_hj+\kappa _h)Z_{j}^{h}(n),\quad n\in \ZZ,
\end{equation}
where, $\eta_h$ and $\kappa _h$ are real numbers independent of $n$ and $j$, and we assume that if for some $n_0$ and $h_0$, $\varepsilon_{n_0}^{h_0}=0$, then also $c_{n_0}=0$ and there is a number $d_{n_0}^{h_0}$  such that the identity (\ref{relaR}) still holds when we replace $c_{n_0}/\varepsilon_{n_0}^{h_0}$ by $d_{n_0}^{h_0}$.

Define the auxiliary numbers $\xi _{n,i}^h$, $i\ge 0$, $n\in \ZZ $ and $h=1,\ldots ,m$, by
\begin{equation*}\label{defxi2}
\xi_{n,i}^h=\prod_{j=n-i+1}^{n}\varepsilon_{j}^{h},\quad i\ge 1,\quad \quad \xi_{n,0}^h=1,\quad\quad \xi_{n,i}^h=\frac{1}{\xi_{n-i,-i}^h},\quad i\leq-1.
\end{equation*}
For $i<0$, we take $\xi_{n,i}^h=\infty$ if $\xi_{n-i,-i}^h=0$ (this can happen if for some $n_0,h_0$, $\varepsilon_{n_0}^{h_0}=0$).
However, by assuming $\varepsilon_n^h\not =0$, $n\le 0$, $h=1,\ldots , m$, one can straightforwardly check that $\xi_{n,n+1}^h$ is finite for $n\le 0$.
Notice that for $x=n$, the number $\xi_{n,i}^h$ coincides with the number defined by (\ref{defxi}).


Given a $m$-tuple $G$ of $m$ positive integers, $G=(g_1,\ldots , g_m)$, assume that the $m$ numbers
\begin{equation}\label{diffn}
\tilde g_h=(\eta_hg_h+\kappa _h),\quad h=1,\ldots ,m,
\end{equation}
are different. Call $\tilde G=\{ \tilde g_1,\ldots , \tilde g_m\}$. Consider finally the $m\times m$ Casorati determinant $\Omega _G$ defined by
\begin{equation*}\label{casort}
\Omega_G (x)=\det \left(\xi_{x,m-j}^lZ_{g_l}^l(\theta_{x-j})\right)_{l,j=1}^m.
\end{equation*}

We then define the sequence of polynomials $(q_n^G)_n$ by
\begin{equation}\label{quso}
q_n^G(x)=\begin{vmatrix}
p_n(x) & -p_{n-1}(x) & \cdots & (-1)^mp_{n-m}(x) \\
\displaystyle\xi_{n,m}^1Z_{g_1}^1(\theta_{n}) & \xi_{n-1,m-1}^1Z_{g_1}^1(\theta_{n-1}) & \cdots &
Z_{g_1}^1(\theta_{n-m}) \\
               \vdots & \vdots & \ddots & \vdots \\
              \xi_{n,m}^mZ_{g_m}^m(\theta_{n}) & \displaystyle
               \xi_{n-1,m-1}^mZ_{g_m}^m(\theta_{n-1}) & \cdots &Z_{g_m}^m(\theta_{n-m})
             \end{vmatrix}.
\end{equation}
Notice that if for each $h=1,\ldots , m$,  $Y_h(x)=Z_{g_h}^h(x)$ is a polynomial in $x$, satisfying the hypotheses of Theorem \ref{Teor1}, then the polynomials $q_n^G$ (\ref{quso})
fit into the definition of the polynomials (\ref{qus})  in Theorem \ref{Teor1}, and hence they are eigenfunctions of an operator in the algebra $\A$.

The key to prove that the polynomials $(q_n^G)_n$ are orthogonal with respect to a measure $\tilde \rho $ are the following formulas.
Assume that $\varepsilon_n^h\not =0$, $n\le 0$, $h=1,\ldots , m,$ and that there exists a constant $c_G\not =0$ such that
\begin{align}\label{foeq1}
\langle \tilde \rho,p_n\rangle &=(-1)^n c_G\sum_{i=1}^{m}\frac{\xi_{n,n+1}^i Z_{g_i}^i(\theta_{n})}{\p _{\tilde G}'(\tilde g_i)Z_{g_i}^{i}(\theta_{-1})},\quad n\ge 0,\\
\label{foeq2}
0&=\sum_{i=1}^{m}\frac{Z_{g_i}^i(\theta_{n})}{\p _{\tilde G}'(\tilde g_i)\xi_{-1,-n-1}^iZ_{g_i}^{i}(\theta_{-1})},\quad 1-m\le n<0,\\
\label{foeq3}
0&\not =\sum_{i=1}^{m}\frac{Z_{g_i}^i(\theta_{-m})}{\p _{\tilde G}'(\tilde g_i)\xi_{-1,m-1}^iZ_{g_i}^{i}(\theta_{-1})},
\end{align}
where $\tilde g_h$ are the $m$ different numbers (\ref{diffn}) and $\p_{\tilde G}(x)=\prod_{i=1}^m(x-\tilde g_i)$.

We then have the following version of Lemma 4.2 of \cite{ddI}:

\begin{lemma}[Lemma 3.4 of \cite{dudh}] \label{lort2} Assume that $\varepsilon_n^h\not =0$, $n\le 0$, $h=1,\ldots , m$, and that there exist $M,N$ (each of them can be either a positive integer or infinity) such that $a_{n}c_n\not =0$ for $1\le n\le N$ and
$\Omega_G (n)\not =0$ for $0\le n\le M$. Assume also that (\ref{foeq1}), (\ref{foeq2}) and (\ref{foeq3}) hold.
Then the polynomials $q_n^G$, $0\le n\le \min\{M-1,N+m\}$, are orthogonal with respect to $\tilde \rho$ and have non-null norms.
\end{lemma}

\subsection{Finite sets of positive integers}
We still need a  last ingredient for identifying the measure $\tilde \rho$ with respect to which the polynomials $(q_n^G)_n$ (\ref{quso}) are orthogonal.
The measures $\rho^{\F}_{a,b,N}$ (\ref{udspmi}) in the Introduction  depends on certain finite sets $F_1, F_2, F_3$ and $F_4$ while the polynomials $(q_n^G)_n$ depend on the finite set $G$ (the degrees of the polynomials $Z$'s). The relationship between the sets $F$'s and $G$ will be given by the following transforms of finite sets of positive integers.

Consider the sets $\Upsilon$ and $\Upsilon _0$ formed by all finite sets of positive or nonnegative integers, respectively:
\begin{align*}
\Upsilon&=\{F:\mbox{$F$ is a finite set of positive integers}\} ,\\
\Upsilon_0&=\{F:\mbox{$F$ is a finite set of nonnegative integers}\} .
\end{align*}
We consider an involution $I$ in $\Upsilon$, and a family $J_h$, $h\ge 1$, of transforms from  $\Upsilon$ into  $\Upsilon _0$. For $F\in \Upsilon$ write $F=\{f_1,\ldots ,f_k\}$ with $f_i<f_{i+1}$, so that $f_k=\max F$. Then $I(F)$ and $J_h(F)$, $h\ge 1$, are defined by
\begin{align}\label{dinv}
I(F)&=\{1,2,\ldots, f_k\}\setminus \{f_k-f,f\in F\},\\ \label{dinv2}
J_h(F)&=\{0,1,2,\ldots, f_k+h-1\}\setminus \{f-1,f\in F\}.
\end{align}

For the involution $I$, the bigger the holes in $F$ (with respect to the set $\{1,2,\ldots , f_k\}$), the bigger the involuted set $I(F)$.
Here it is a couple of examples
$$
I(\{ 1,2,3,\ldots ,k\})=\{ k\},\quad \quad I(\{1, k\})=\{ 1,2,\ldots, k-2, k\}.
$$
Something similar happens for the transform $J_h$ with respect to $\{0,1,\ldots , f_k+h-1\}$.

Notice that
$$
\max F=\max I(F), \quad h-1+\max F=\max J_h(F),
$$
and if $n_F$ denotes the cardinal of $F$, we also have
\begin{equation}\label{neci}
n_{I(F)}=f_k-n_F+1,\quad n_{J_h(F)}=f_k+h-n_F.
\end{equation}

For a quartet $\mathcal F=(F_1,F_2,F_3,F_4)$ of finite sets of positive integers, we will write $F_i=\{f_1^{i\rceil},\ldots,f_{n_F}^{i\rceil}\}, i=1,2,3,4,$ with $f_j^{i\rceil}<f_{j+1}^{i\rceil}$ (the use, for instance, of $f_j^2$ to describe elements of $F_2$ is confusing because it looks like a square. This is the reason why we use the notation $f_j^{i\rceil}$).

\section{Hahn polynomials}\label{SEC4}
We start with some basic definitions and facts about Hahn and dual Hahn  polynomials, which we will need later.
For $a,a+b+1,a+b+N+1\neq-1,-2,\ldots$ we write $(h_n^{a,b,N})_n$ for the sequence of Hahn polynomials defined by
\begin{equation}\label{HP}
h_n^{a,b,N}(x)=\sum_{j=0}^n\frac{(-x)_j(N-n+1)_{n-j}(a+b+1)_{j+n}}{(2+a+b+N)_n(a+1)_j(n-j)!j!},\quad n\ge 0,
\end{equation}
(we use a different normalization from the one used in \cite{du1}, pp. 35. The equivalence is given by $\frac{(-1)^nn!(a+1)_n(2+a+b+N)_n}{(a+b+1)_{2n}}h_n^{a,b,N}(x)=\tilde h_n^{b+N+1,a+1,N+1}(x)$, where $(\tilde h_n^{\alpha,c,N})_n$ is the family used in \cite{du1}).

When $N$ is a positive integer then the polynomial $h_{n}^{a,b,N}(x)$ for $n\ge N+1$ is always divisible by $(-x)_{N+1}$. Hence
\begin{equation}\label{mdc1}
h_{n}^{a,b,N}(i)=0,\quad n\ge N+1,\quad i=0,\ldots, N.
\end{equation}
Hahn polynomials are eigenfunctions of the following second-order difference operator
\begin{equation}\label{dopHahn}
D_{a,b,N}=x(x-b-N-1)\Sh_{-1}-\big[(x+a+1)(x-N)+x(x-b-N-1)\big]\Sh_{0}+(x+a+1)(x-N)\Sh_{1}.
\end{equation}
That is
\begin{equation*}\label{autovH}
D_{a,b,N}(h_n^{a,b,N})=\theta_nh_n^{a,b,N},\quad\theta_n=n(n+a+b+1), \quad n\geq0.
\end{equation*}

They satisfy the following three-term recurrence formula ($h_{-1}^{a,b,N}=0$)
\begin{equation}\label{ttrrH}
    xh_n=a_{n+1}h_{n+1}+b_nh_n+c_nh_{n-1},\quad n\geq0,
\end{equation}
where\begin{align*}
       a_n &= -\frac{n(n+a)(n+a+b+N+1)}{(2n+a+b-1)(2n+a+b)}, \\
       b_n &= \frac{N(a+1)(a+b)+n(2N+b-a)(n+a+b+1)}{(2n+a+b)(2n+a+b+2)}, \\
       c_n &= -\frac{(n+a+b)(n+b)(N-n+1)}{(2n+a+b)(2n+a+b+1)}
     \end{align*}
(to simplify the notation we remove the parameters in some formulas).

Assume that $a,b,a+b,a+b+N+1\neq -1,-2,\ldots,$ and $N+1$ is not a positive integer, then the Hahn polynomials are always orthogonal with respect to a moment functional $\rho_{a,b,N}$ which we normalize by taking
$$
\langle\rho_{a,b,N},1\rangle=\frac{\Gamma(a+b+N+2)\Gamma(a+1)\Gamma(b+1)}{\Gamma(a+b+2)}.
$$
When $N+1$ is a positive integer, $a,b\neq -1,-2,\ldots,-N$, and $a+b\neq -1,-2,\ldots,-2N-1$, the first $N+1$ Hahn polynomials are orthogonal with respect to the Hahn measure
\begin{equation}\label{pesoH}
    \rho_{a,b,N}=N!\sum_{x=0}^N\frac{\Gamma(a+x+1)\Gamma(N-x+b+1)}{x!(N-x)!}\delta_x,
\end{equation}
and have non-null norms. The discrete measure $\rho_{a,b,N}$ is positive only when $a,b>-1$ or $a,b<-N$.

We also need the so-called dual Hahn polynomials, $a\not =-1,-2,\ldots $, $n\geq0$,
\begin{equation}\label{dHP}
   R_n^{a,b,N}(x)=\sum_{j=0}^n\frac{(-1)^j(-n)_j(-N+j)_{n-j}}{(a+1)_jj!}\prod_{i=0}^{j-1}[x-i(i+a+b+1)].
\end{equation}
Observe that $(-1)^j\prod_{i=0}^{j-1}(x(x+a+b+1)-i(a+b+1+i))=(-x)_j(x+a+b+1)_j$, therefore we have the duality
$$
R_x^{a,b,N}(n(n+a+b+1))=\frac{(-1)^nn!(N+a+b+2)_n(-N)_x}{(a+b+1)_n(-N)_n}h_n^{a,b,N}(x), \quad x,n\geq0.
$$

\medskip

Consider now the algebra of differential operators defined by \eqref{algdiffd}. There are 4 different $\D$-operators for the Hahn polynomials (see Lemma 7.2 of \cite{du1}).
They are defined by the sequences $(\varepsilon_{n,h})_n$ and $(\sigma_n)_n, h=1,2,3,4,$ given by

\begin{align}
\label{eps1}&\varepsilon_{n,1}=-\frac{n-N+1}{n+a+b+N+1},\quad&\sigma_{n}=-(2n+a+b-1),\\
\label{eps2}&\varepsilon_{n,2}=\frac{(n+b)(n-N+1)}{(n+a)(n+a+b+N+1)},\quad&\sigma_{n}=-(2n+a+b-1),\\
\label{eps3}&\varepsilon_{n,3}=1,\quad&\sigma_{n}=-(2n+a+b-1),\\
\label{eps4}&\varepsilon_{n,4}=-\frac{n+b}{n+a},\quad&\sigma_{n}=-(2n+a+b-1).
\end{align}
These sequences define four $\mathcal{D}$-operators (see \eqref{Dh}):
\begin{align}\label{dohh1}
&\mathcal{D}_1=\frac{a+b+1}{2}I+x\nabla,&\quad&\mathcal{D}_2=\frac{a+b+1}{2}I+(x-N)\Delta,\\\label{dohh2}
&\mathcal{D}_3=\frac{a+b+1}{2}I+(x+a+1)\Delta,&\quad&\mathcal{D}_4=\frac{a+b+1}{2}I+(x-b-N-1)\nabla.
\end{align}
$\Delta$ and $\nabla$ denote, as usual, the first-order difference operators:
$$
\Delta(f)=f(x+1)-f(x),\quad \nabla(f)=f(x)-f(x-1).
$$

Let us call $N_{x}^{h;j}$ and $D_{x}^{h;j}, h=1,2,$ the following functions:
\begin{align*}\label{ABDijH}
N_{x}^{1;j}&=(x-j+b+1)_j \quad &N_{x}^{2;j}&=(x-j-N)_j,\\
D_{x}^{1;j}&=(-1)^j(x-j+a+1)_j \quad &D_{x}^{2;j}&=(-1)^j(x-j+a+b+N+2)_j.
\end{align*}
We will use the following properties, which easily hold by definition
\begin{equation}\label{ABprop2}
N_{x-i}^{h;m-i}=N_{x-i}^{h;j}N_{x-i-j}^{h;m-i-j},\quad D_{x-i}^{h;m-i}=D_{x-i}^{h;j}D_{x-i-j}^{h;m-i-j},\quad h=1,2.
\end{equation}

Given a quartet $\U=(U_1,U_2,U_3,U_4)$  of finite sets of nonnegative integers, we write $m_j$
for the cardinal of $U_j$, $j=1,2,3,4$, $m=m_1+m_2+m_3+m_4$ and
\begin{align}\label{ppmm}
\UU_1&=\{1,\ldots, m_1\},\quad &\UU_2&=\{m_1+1,\ldots, m_1+m_2\},
\\\nonumber \UU_3&=\{m_1+m_2+1,\ldots, m_1+m_2+m_3\}, \quad &\UU_4&=\{m_1+m_2+m_3+1,\ldots, m\}.
\end{align}
We write $U_j=\{u_i^{j\rceil},i\in \UU_j\}$ (see the end of Section 4.2 for a justification of this notation).

The functions $\xi_{x,j}^h$ defined in \eqref{defxi} can be written as

\begin{equation}\label{xit}
\xi_{x,j}^h=\begin{cases} \displaystyle\frac{N_{x}^{2;j}}{D_{x}^{2;j}}=(-1)^j\frac{(x-j-N)_j}{(x-j+a+b+N+2)_j},& \mbox{for $h\in \UU_1$,}\\
\displaystyle\frac{N_{x}^{1;j}N_{x}^{2;j}}{D_{x}^{1;j}D_{x}^{2;j}}=\frac{(x-j+b+1)_j(x-j-N)_j}{(x-j+a+1)_j(x-j+a+b+N+2)_j},& \mbox{for $h\in \UU_2$,}\\
1,& \mbox{for $h\in \UU_3$,}\\
\displaystyle\frac{N_{x}^{1;j}}{D_{x}^{1;j}}=(-1)^j\frac{(x-j+b+1)_j}{(x-j+a+1)_j},& \mbox{for $h\in \UU_4$.}
 \end{cases}
\end{equation}

\medskip

In the rest of this Section, we will prove that the three hypotheses (\ref{ass0}), (\ref{ass1}) and (\ref{ass2}) in Theorem \ref{Teor1}
hold for the four $\D$-operators above associated to the Hahn polynomials. Hence, the polynomials $(q_n)_n$ defined by \eqref{qus} for the Hahn family will be consequently eigenfunctions of the higher-order difference operator \eqref{Dq}.

To check the first hypothesis (\ref{ass0}) in Theorem \ref{Teor1}, we will need the following lemma,
which it will be also useful to compute the order of the difference operator with respect to which the bispectral polynomials constructed in Section \ref{sch} will be eigenfunctions.

\begin{lemma}\label{lgp1} Let $Y_1, Y_2, \ldots, Y_m,$ be nonzero polynomials satisfying that $\deg Y_i=u_{i}^{j\rceil}$, if $i\in \UU_j$ and $1\le j\le 4$. Write $r_i$ for the leading coefficient of $Y_i$, $1\le i\le m$.
For real numbers $a,b, N$, consider the rational function $P$ defined by
\begin{equation}\label{def2p}
P(x)=\frac{\left|
  \begin{array}{@{}c@{}lccc@{}c@{}}
    &&&\hspace{-.9cm}{}_{j=1,\ldots,m} \\
    \dosfilas{N_{x-j}^{2;m-j}D_{x-1}^{2;j-1}Y_{i}(\theta _{x-j}) }{i\in \UU_1} \\
    \dosfilas{N_{x-j}^{1;m-j}N_{x-j}^{2;m-j}D_{x-1}^{1;j-1}D_{x-1}^{2;j-1}Y_{i}(\theta _{x-j}) }{i\in \UU_2}\\
     \dosfilas{Y_{i}(\theta _{x-j})}{i\in \UU_3}\\
    \dosfilas{N_{x-j}^{1;m-j}D_{x-1}^{1;j-1}Y_{i}(\theta _{x-j}) }{i\in \UU_4}
  \end{array}
   \hspace{-.4cm}\right|}{p(x)q(x)},
\end{equation}
where $p$ and $q$ are the polynomials
\begin{align}\label{def1p}
p(x)&=\prod_{i=1}^{m_2+m_4-1}N^{1;m_2+m_4-i}_{x-m_1-m_3-i}D^{1;m_2+m_4-i}_{x-1}\prod_{i=1}^{m_1+m_2-1}N^{2;m_1+m_2-i}_{x-m_3-m_4-i}
D^{2;m_1+m_2-i}_{x-1},\\
\label{def1q}q(x)&=(-1)^{\frac{m(m-1)}{2}}\prod_{p=1}^{m-1}\left(\prod_{s=1}^{p}\sigma_{x-m+\frac{s+p+1}{2}}\right).
\end{align}
The determinant \eqref{def2p} should be understood in the way explained in the Preliminaries (see \eqref{defdosf}). If
\begin{equation}\label{yas0}
z-u+b+N+1, w-v+a+N+1\not =0,
\end{equation}
for $u\in U_1,v\in U_2,w\in U_3,z\in U_4$, then $P$ is a polynomial of degree
\begin{equation*}\label{ddp}
d=2\sum_{u\in U_1,U_2,U_3,U_4}u-2\sum_{i=1}^4\binom{m_i}{2},
\end{equation*}
with leading coefficient given by
\begin{align*}\label{mspcl}
r&=(-1)^{\sum_{i=1}^{4}\binom{m_i}{2}+m_1m_2+m_2m_3+m_3m_4}V_{U_1}V_{U_2}V_{U_3}V_{U_4}\prod_{i=1}^mr_i
\\\nonumber&\qquad \times\prod_{v\in U_2,w\in U_3}(N+a+1-v+w)\prod_{u\in U_1,z\in U_4}(N+b+1-u+z),
\end{align*}
where $V_X$ denotes the Vandermonde determinant associated to the set $X=\{ x_1,x_2,\ldots,x_K\}$ defined by $V_X=\displaystyle\prod_{i<j}(x_j-x_i)$.
\end{lemma}

\begin{proof}
The Lemma can be proved using the same approach as in the proof of Lemma 3.3 in \cite{dudh}.
\end{proof}

\bigskip

Let us now introduce the key concept in order to check the hypotheses (\ref{ass1}) and (\ref{ass2}) in Theorem \ref{Teor1} for the Hahn polynomials. We define an \emph{involution} that characterizes the subring $\mathbb{R}[\theta_x]$ in $\mathbb{R}[x]$. This involution is given by
\begin{equation}\label{inv}
\I^{a+b}\big(p(x)\big)=p\big(-(x+a+b+1)\big),\quad p\in\mathbb{R}[x].
\end{equation}
Clearly we have $\I^{a+b}(\theta_x)=\theta_x$. Hence every polynomial in $\theta_x$ is invariant under the action of $\I^{a+b}$. Conversely, if $p\in\mathbb{R}[x]$ is invariant under $\I^{a+b}$, then $p\in\mathbb{R}[\theta_x]$. We also have that if $p\in\mathbb{R}[x]$ is skew invariant, i.e. $\I^{a+b}(p)=-p,$ then $p$ is divisible by $\theta_{x-1/2}-\theta_{x+1/2}$, and the quotient belongs to $\mathbb{R}[\theta_x]$. We remark here that, in the case of Hahn polynomials and from the definition of $\theta_x$ and $\sigma_x$, we have that $\sigma_{x+1}=\theta_{x-1/2}-\theta_{x+1/2}$. Observe that $\sigma_{x+1}$ is skew invariant itself.

We have the following properties according to the definition \eqref{inv}, $h=1,2$:
\begin{align}
\label{Iprop1}\I^{a+b+i}(\theta_{x-j})&=\theta_{x+i+j},\quad &\I^{a+b+i}(\sigma_{x-j})&=-\sigma_{x+i+j+2},\\
\label{ABpropsH}
\I^{a+b+i}(N_{x-j-s}^{h;m-s})&=D_{x+m+i+j}^{h;m-s},\quad &\I^{a+b+i}(D_{x-j-s}^{h;m-s})&=N_{x+m+i+j}^{h;m-s}.
\end{align}

\medskip

We are now ready to check that the three hypotheses (\ref{ass0}), (\ref{ass1}) and (\ref{ass2}) in Theorem \ref{Teor1}
hold for the four $\D$-operators above associated to the Hahn polynomials.

\begin{lemma}\label{l5.2} Let $\A$ and $(p_n)_n$ be respectively, the algebra of difference operators (\ref{algdiffd}) and the sequence of Hahn polynomials $p_n=h_n^{a,b, N}$. We denote by $D_p$ the second-order difference operator (\ref{dopHahn}), so that $\theta_n=n(n+a+b+1)$ and
$D_p(p_n)=\theta_np_n$. For $j=1,2,3,4$, we also have $m_j$ $\D$-operators defined by the sequences $(\varepsilon_{n,j})_n$, $(\sigma_n)_n$ (see
(\ref{eps1})--(\ref{eps4})). Write then $m=\sum_{i=1}^4 m_i$ and let $\Xi$ be a  polynomial in $x$ invariant under the action of $\I^{a+b-m-1}$. Define  the rational function $S$  by
\begin{equation}\label{GGH}
S(x)=\sigma_{x-\frac{m-1}{2}}\Xi(x)\frac{(D^{1;m-1}_{x-1})^{m_2+m_4}(D^{2;m-1}_{x-1})^{m_1+m_2}}{p(x)q(x)},
\end{equation}
where $p$ and $q$ are the polynomials defined by (\ref{def1p}) and (\ref{def1q}), respectively.
Then the three hypotheses (\ref{ass0}), (\ref{ass1}) and (\ref{ass2}) in Theorem \ref{Teor1}
hold.
\end{lemma}

\begin{proof}
Consider the sets $\UU_j$, $j=1,2,3,4$, given by (\ref{ppmm}). By interchanging rows, we can assume that each $\UU_j$
is formed by the indexes $h$ where the $\mathcal{D}$-operator $\D_h$ is
defined by the sequence $(\varepsilon_{n,j})_n$ (see (\ref{eps1})--(\ref{eps4})).

Since the polynomial $\Xi$ is invariant under the action of $\I^{a+b-m-1}$, we have
\begin{equation}\label{Qprop}
\I^{a+b+i}\big(\Xi(x-j)\big)=\Xi(x+m+i+j+1).
\end{equation}
As a consequence of \eqref{Iprop1} and \eqref{ABpropsH} we have
\begin{equation}\label{qprop}
\I^{a+b+i}(q(x-j))=(-1)^{\frac{m(m-1)}{2}}q(x+i+j+m+1),
\end{equation}
and
\begin{equation}\label{pprop}
\I^{a+b+i}(p(x-j))=p(x+i+j+m+1).
\end{equation}

\medskip

We now check the first assumption (\ref{ass0}) in Theorem \ref{Teor1}, that is:
$S(x)\Omega (x)$ is a polynomial in $x$.

From the definition of $S(x)$ in \eqref{GGH} and $\Omega(x)$ in \eqref{casd1} it is straightforward to see,
using \eqref{xit} and \eqref{ABprop2}, that

\begin{align}
\label{SOm}& S(x)\Omega(x)=\frac{\sigma_{x-\frac{m-1}{2}}\Xi(x)}{p(x)q(x)}\left|
 \begin{array}{@{}c@{}lccc@{}c@{}}
        &&&\hspace{-1.5cm}{}_{j=1,\ldots , m} \\
        \dosfilas{N_{x-j}^{2;m-j}D_{x-1}^{2;j-1}Y_l(\theta_{x-j})}{l\in \UU_1}\\
    \dosfilas{N_{x-j}^{1;m-j}N_{x-j}^{2;m-j}D_{x-1}^{1;j-1}D_{x-1}^{2;j-1}Y_l(\theta_{x-j})}{l\in \UU_2} \\
    \dosfilas{Y_l(\theta_{x-j})}{l\in \UU_3} \\
    \dosfilas{N_{x-j}^{1; m-j}D_{x-1}^{1;j-1}Y_l(\theta_{x-j})}{l\in \UU_4}
  \end{array}
  \right|.
\end{align}
Therefore $S(x)\Omega(x)=\sigma_{x-\frac{m-1}{2}}\Xi (x)P(x)$, where $P$ is the rational function (\ref{def2p}) defined in the Lemma \ref{lgp1}. According to this Lemma, $P$ is actually a polynomial and hence $S(x)\Omega(x)$ is a polynomial as well.

\medskip

We now check the second assumption (\ref{ass1}) in Theorem \ref{Teor1}, that is:
there exist $\tilde{M}_1,\ldots,\tilde{M}_m,$ polynomials in $x$ such that
\begin{equation*}
M_h(x)=\sigma_{x+1}\tilde{M}_h(\theta_x),\quad h=1,\ldots,m.
\end{equation*}

Write now
$$
\Psi_j^h(x)=\xi_{x-j,m-j}^hS(x)
\det\left(\xi_{x-r,m-r}^{l}Y_l(\theta_{x-r})\right)_{l\in \II_h;r\in \II_j}, \quad h,j=1,\ldots , m.
$$
A simple computation using (\ref{xit}) shows that $\Psi_j^h$, $h,j=1,\ldots , m$, are polynomials in $x$. Hence Lemma \ref{lgp2} gives that $M_h$ is also a polynomial in $x$.

It is now enough to see that
\begin{equation*}\label{InvMh}
\I^{a+b}(M_h(x))=-M_h(x),\quad h=1,\ldots,m,
\end{equation*}
where $M_h(x), h=1,\ldots,m,$ are defined in \eqref{emeiexp}. Hence, $M_h(x), h=1,\ldots,m,$ according to the discussion after \eqref{inv}, is divisible by $\sigma_{x+1}$ and the quotient belongs to $\mathbb{R}[\theta_x]$.

Assume that the $h$-th $\mathcal{D}$-operator is $\mathcal{D}_1$ (similar for $\mathcal{D}_2$, $\mathcal{D}_3$ and $\mathcal{D}_4$). In that case, as before, we can remove all denominators in $M_h(x)$ and rearrange the determinant to write
\begin{align*}
\nonumber M_h(x)=\sum_{j=1}^m&(-1)^{h+j}\frac{\sigma_{x+j-\frac{m-1}{2}}\Xi(x+j)}{p(x+j)q(x+j)}N_{x}^{2;m-j}D_{x+j-1}^{2;j-1}\times\\
&\qquad \times\left|
  \begin{array}{@{}c@{}lccc@{}c@{}}
        &&&\hspace{-1.5cm}{}_{l\neq h, r\neq j} \\
        \dosfilas{N_{x+j-r}^{2;m-r}D_{x+j-1}^{2;r-1}Y_l(\theta_{x+j-r})}{l\in \UU_1\setminus\{h\}}\\
    \dosfilas{N_{x+j-r}^{1;m-r}N_{x+j-r}^{2;m-r}D_{x+j-1}^{1;r-1}D_{x+j-1}^{2;r-1}Y_l(\theta_{x+j-r})}{l\in \UU_2} \\
    \dosfilas{Y_l(\theta_{x+j-r})}{l\in \UU_3} \\
    \dosfilas{N_{x+j-r}^{1;m-r}D_{x+j-1}^{1;r-1}Y_l(\theta_{x+j-r})}{l\in \UU_4}
  \end{array}
  \right|.
\end{align*}
Hence, using \eqref{ABpropsH}, \eqref{Qprop}, \eqref{qprop} and \eqref{pprop}, we have
\begin{align*}
&\I^{a+b}\big(M_h(x)\big)=-\sum_{j=1}^m(-1)^{h+j}(-1)^{\frac{m(m-1)}{2}}\frac{\sigma_{x+\frac{m-1}{2}-j+2}\Xi(x+m-j+1)}{p(x+m-j+1)q(x+m-j+1)}\times\\
&\hspace{0.5cm}\qquad\qquad\qquad\times D_{x+m-j}^{2;m-j}N_{x}^{2;j-1}\left|
 \begin{array}{@{}c@{}lccc@{}c@{}}
        &&&\hspace{-1.5cm}{}_{l\neq h, r\neq j} \\
        \dosfilas{D_{x+m-j}^{2;m-r}N_{x+r-j}^{2;r-1}Y_l(\theta_{x-j+r})}{l\in \UU_1\setminus\{h\}}\\
    \dosfilas{D_{x+m-j}^{1;m-r}D_{x+m-j}^{2;m-r}N_{x+r-j}^{1;r-1}N_{x+r-j}^{2;r-1}Y_l(\theta_{x-j+r})}{l\in \UU_2} \\
    \dosfilas{Y_l(\theta_{x-j+r})}{l\in \UU_3} \\
    \dosfilas{D_{x+m-j}^{1;m-r}N_{x+r-j}^{1; r-1}Y_l(\theta_{x-j+r})}{l\in \UU_4}
  \end{array}
  \right|\\
 &=-(-1)^{\frac{m(m-1)}{2}}(-1)^{m-1}(-1)^{\frac{(m-1)(m-2)}{2}}\sum_{j=1}^m(-1)^{h+j}\frac{\sigma_{x+j-\frac{m-1}{2}}\Xi(x+j)}{p(x+j)q(x+j)} \times\\
&\hspace{0.5cm}\qquad\qquad\qquad\times D_{x+j-1}^{2;j-1}N_{x}^{2;m-j}\left|
 \begin{array}{@{}c@{}lccc@{}c@{}}
        &&&\hspace{-1.5cm}{}_{l\neq h, r\neq j} \\
        \dosfilas{D_{x+j-1}^{2;r-1}N_{x+j-r}^{2;m-r}Y_l(\theta_{x+j-r})}{l\in \UU_1\setminus\{h\}}\\
    \dosfilas{D_{x+j-1}^{1;r-1}D_{x+j-1}^{2;r-1}N_{x+j-r}^{1;m-r}N_{x+j-r}^{2;m-r}Y_l(\theta_{x+j-r})}{l\in \UU_2} \\
    \dosfilas{Y_l(\theta_{x+j-r})}{l\in \UU_3} \\
    \dosfilas{D_{x+j-1}^{1;r-1}N_{x+j-r}^{1; m-r}Y_l(\theta_{x+j-r})}{l\in \UU_4}
  \end{array}
  \right|\\
&=-(-1)^{m(m-1)}M_h(x)=-M_h(x).
\end{align*}
Observe that in the second step we have renamed the index $j$ while in the determinant we have interchanged all columns, so we have the corresponding change of signs.

\medskip

We finally check the third assumption (\ref{ass2}) in Theorem \ref{Teor1}, that is: there exists a polynomial $P_S$ such that
\begin{equation*}
P_S(\theta_x)=2\lambda_x+\sum_{h=1}^mY_h(\theta_x)M_h(x).
\end{equation*}
Call
$$
H(x)=2\lambda_x+\sum_{h=1}^mY_h(\theta_x)M_h(x).
$$
Taking into account that $M_h$, $h=1,\ldots , m$, are polynomials in $x$, we conclude that $H$ is also a polynomial in $x$.
From \eqref{Pdiff} we have that
$$
H(x)-H(x-1)=S(x)\Omega(x)+S(x+m)\Omega(x+m).
$$
For this step it is enough to see that
\begin{equation}\label{InvGO}
\I^{a+b-1}\big(H(x)-H(x-1)\big)=-\big(H(x)-H(x-1)).
\end{equation}
If such is the case, then we can always find a polynomial $P_S$ such that $P_S(\theta_x)=H(x)$, since we already know, from \eqref{Iprop1}, that
$$
\I^{a+b-1}\big(P_S(\theta_x)-P_S(\theta_{x-1})\big)=P_S(\theta_{x-1})-P_S(\theta_x)=-\big(P_S(\theta_x)-P_S(\theta_{x-1})\big).
$$

To proof \eqref{InvGO} it is enough to see that
$$
\I^{a+b-1}\big(S(x)\Omega(x)\big)=-\big(S(x+m)\Omega(x+m)\big).
$$
From \eqref{SOm} we have, using again \eqref{ABpropsH}, \eqref{Qprop}, \eqref{qprop} and \eqref{pprop}, that

\begin{align*}
& \I^{a+b-1}\big(S(x)\Omega(x)\big)=-(-1)^{\frac{m(m-1)}{2}}\frac{\sigma_{x+\frac{m+1}{2}}\Xi(x+m)}{p(x+m)q(x+m)}\times \\
&\hspace{0.5cm}\qquad\qquad\qquad\times \left|
 \begin{array}{@{}c@{}lccc@{}c@{}}
        &&&\hspace{-1.5cm}{}_{j=1,\ldots , m} \\
        \dosfilas{D_{x+m-1}^{2;m-j}N_{x+j-1}^{2;j-1}Y_l(\theta_{x+j-1})}{l\in \UU_1}\\
    \dosfilas{D_{x+m-1}^{1;m-j}D_{x+m-1}^{2;m-j}N_{x+j-1}^{1;j-1}N_{x+j-1}^{2;j-1}Y_l(\theta_{x+j-1})}{l\in \UU_2} \\
    \dosfilas{Y_l(\theta_{x+j-1})}{l\in \UU_3} \\
    \dosfilas{D_{x+m-1}^{1;m-j}N_{x+j-1}^{1; j-1}Y_l(\theta_{x+j-1})}{l\in \UU_4}
  \end{array}
  \right|\\
=&-(-1)^{\frac{m(m-1)}{2}}(-1)^{\frac{m(m-1)}{2}}\frac{\sigma_{x+m-\frac{m-1}{2}}\Xi(x+m)}{p(x+m)q(x+m)}\times \\
&\hspace{0.5cm}\qquad\qquad\qquad\times\left|
 \begin{array}{@{}c@{}lccc@{}c@{}}
        &&&\hspace{-1.5cm}{}_{j=1,\ldots , m} \\
        \dosfilas{D_{x+m-1}^{2;j-1}N_{x+m-j}^{2;m-j}Y_l(\theta_{x+m-j})}{l\in \UU_1}\\
    \dosfilas{D_{x+m-1}^{1;j-1}D_{x+m-1}^{2;j-1}N_{x+m-j}^{1;m-j}N_{x+m-j}^{2;m-j}Y_l(\theta_{x+m-j})}{l\in \UU_2} \\
    \dosfilas{Y_l(\theta_{x+m-j})}{l\in \UU_3} \\
    \dosfilas{D_{x+m-1}^{1;j-1}N_{x+m-j}^{1; m-j}Y_l(\theta_{x+m-j})}{l\in \UU_4}
  \end{array}
  \right|\\
=&-S(x+m)\Omega(x+m).
\end{align*}

\end{proof}

\section{Bispectral Krall-Hahn polynomials}\label{sch}
In this Section we put together all the ingredients showed in the previous Sections to construct bispectral Krall-Hahn polynomials.

We can apply Theorem \ref{Teor1} to produce from arbitrary polynomials $Y_j$, $j\ge 0$, a large class of sequences of polynomials $(q_n)_n$ satisfying higher-order difference equations. But only for a convenient choice of the polynomials $Y_j$, $j\ge 0$, these polynomials $(q_n)_n$ are also orthogonal with respect to a measure. As we wrote in the Introduction, when the sequence $(p_n)_n$ is the Hahn polynomials, a very nice symmetry between the family $(p_n)_n$ and the polynomials $Y_j$'s appears. Indeed, the polynomials $Y_j$ can be chosen as dual Hahn polynomials with parameters depending on the $\D$-operator $\D_h$. This symmetry is given by the recurrence relation (\ref{relaR}), where $Y_i=Z_i^h$.

\begin{lemma}\label{lemRme} Consider the dual Hahn polynomials
\begin{align*}
Z_{j}^{1}(x)&=R_j^{-b,-a,a+b+N}(x+a+b),\quad j\ge 0,\\
Z_{j}^{2}(x)&=R_j^{-a,-b,a+b+N}(x+a+b),\quad j\ge 0,\\
Z_{j}^{3}(x)&=R_j^{-b,-a,-2-N}(x+a+b),\quad j\ge 0,\\
Z_{j}^{4}(x)&=R_j^{-a,-b,-2-N}(x+a+b),\quad j\ge 0.
\end{align*}
Then they satisfy the recurrence (\ref{relaR}), where $(a_n)_{n\in \ZZ }$, $(b_n)_{n\in \ZZ }$, $(c_n)_{n\in \ZZ }$ are the sequences of coefficients in the three-term recurrence relation for the  Hahn polynomials $(h_n^{a,b,N})_n$ (\ref{ttrrH}) and
\begin{align*}
&\varepsilon_{n}^{1}=-\frac{n-N-1}{n+a+b+N+1}, \qquad &\eta_1&=1,\quad &\kappa _1&=-b-N,\\
&\varepsilon_{n}^{2}=\frac{(n-N-1)(n+b)}{(n+a)(n+a+b+N+1)}, \qquad &\eta_2&=-1,\quad &\kappa _2&=a,\\
&\varepsilon_{n}^{3}=1, \qquad &\eta_3&=-1,\quad &\kappa _3&=-N-1,\\
&\varepsilon_{n}^{4}=-\frac{n+b}{n+a}, \qquad &\eta_4&=1,\quad &\kappa _4&=1.\\
\end{align*}
\end{lemma}

\begin{proof}
It is similar to the proof of Lemma 5.1 of \cite{ddI}.
\end{proof}
From now on we will assume that $N$ is a positive integer. This condition is necessary for the existence of a positive weight for the Hahn polynomials, and only in this case we have an explicit expression of that weight. However, this condition is not needed in our construction and hence Theorem \ref{t6.2} and Corollary \ref{jodme} are also valid when $N$ is not a positive integer (once one has adapted the constraints on the parameters $a$ and $b$).

Since we  have four $\D$-operators for Hahn polynomials, we make a partition of the indices in Theorem \ref{Teor1} (see \eqref{ppmm})
and take
\begin{equation}\label{defvardh}
\varepsilon _n^h=\begin{cases} -\displaystyle\frac{n-N-1}{n+a+b+N+1},& \mbox{for $h\in \UU_1$,}\\
\displaystyle\frac{(n-N-1)(n+b)}{(n+a)(n+a+b+N+1)},& \mbox{for $h\in \UU_2$,}\\
1,& \mbox{for $h\in \UU_3$,}\\
-\displaystyle\frac{n+b}{n+a},& \mbox{for $h\in \UU_4$.}
 \end{cases}
\end{equation}
For the sequences $(\sigma_n^h)_n$, we always take $\sigma_n^h=\sigma_n=-(2n+a+b-1)$.
In particular, the auxiliary functions $\xi_{x,i}^h$, $h=1,\ldots, m$, $i\in \ZZ$, (see (\ref{defxi})) are then defined by \eqref{xit}. Finally write
$$
\ZZ_{i}=\{ j\in \ZZ:j\le i\},\quad i\in \ZZ.
$$

We are now ready to establish  the main Theorem of this paper.

\begin{theorem}\label{t6.2} Let $\F=(F_1,F_2,F_3,F_4)$ be a quartet of finite sets of positive integers (the empty set is allowed, in which case we take $\max F=-1$). For $\h=(h_1,h_2,h_3)$, $h_i\ge 1$, consider the quartet $\U=(U_1,U_2,U_3,U_4)$ whose elements are the transformed sets $J_{h_j}(F_j)=U_j=\{ u_{i}^{j\rceil}: i\in \UU_j\}$, $j=1,2,3$, and $I(F_4)=U_4=\{ u_{i}^{4\rceil}: i\in \UU_4\}$, where the involution $I$ and the transform $J_h$ are defined by (\ref{dinv}) and (\ref{dinv2}), respectively. Define $m=m_1+m_2+m_3+m_4$, where $m_i$ denotes de cardinal of $U_i$, $f_{i,M}=\max (F_i),$ and $n_{F_i}$ the cardinal of $F_i$ for $i=1,2,3,4$.
Let $a$ and $b$ be real numbers satisfying
\begin{equation}\label{appm1}
a\not \in \ZZ_{f_{2,M}+f_{4,M}+h_2}, \quad b\not \in \ZZ_{f_{1,M}+f_{3,M}+h_1+h_3-1},\quad a+b\not\in \ZZ_{\sum_{i=1}^4f_{i,M}+\sum_{i=1}^3h_i}.
\end{equation}
In addition, we assume that
\begin{equation}\label{appm2}
a\not=1,2,\ldots, \quad \mbox{if \quad$F_2$ or $F_4\not =\emptyset$,\quad and \quad  $b\not=1,2,\ldots,$ \quad if \quad $F_1$ or $F_3\not =\emptyset$}.
\end{equation}
Consider the Hahn and dual Hahn polynomials $(h_n^{a,b,N})_n$ (\ref{HP}) and $(R_n^{a,b,N})_n$ (\ref{dHP}), respectively. Assume that $\Omega_{a,b,N}^{\U,\h} (n)\not =0$ for $0\le n\le N+m_3+m_4+1$ where the $m\times m$ Casorati determinant $\Omega _{a,b,N}^{\U,\h}$ is defined by
\begin{align*}\label{defom}
\Omega _{a, b, N}^{\U,\h}(x)&=  \left|
  \begin{array}{@{}c@{}lccc@{}c@{}}
    & &&\hspace{-1.3cm}{}_{j=1,\ldots , m} \\
    \dosfilas{ \xi_{x-j,m-j}^1R_{u}^{-b,-a,a+b+N}(\theta_{x-j}+a+b) }{u\in U_1} \\
    \dosfilas{ \xi_{x-j,m-j}^2R_{u}^{-a,-b,a+b+N}(\theta_{x-j}+a+b)}{u\in U_2}
    \\
    \dosfilas{ \xi_{x-j,m-j}^3R_{u}^{-b,-a,-2-N}(\theta_{x-j}+a+b) }{u\in U_3}\\
    \dosfilas{\xi_{x-j,m-j}^4R_{u}^{-a,-b,-2-N}(\theta_{x-j}+a+b) }{u\in U_4}
  \end{array}\right|.
\end{align*}
We then define the sequence of polynomials $q_n$, $n\ge 0$, by
\begin{equation}\label{qusch}
q_n(x)=\left|
  \begin{array}{@{}c@{}lccc@{}c@{}}
   &(-1)^{j-1}h_{n+1-j}^{a,b,N}(x) &&\hspace{-1.3cm}{}_{j=1,\ldots , m+1} \\
    \dosfilas{ \xi_{n-j,m-j}^1R_{u}^{-b,-a,a+b+N}(\theta_{n-j}+a+b) }{u\in U_1} \\
    \dosfilas{ \xi_{n-j,m-j}^2R_{u}^{-a,-b,a+b+N}(\theta_{n-j}+a+b)}{u\in U_2}
    \\
    \dosfilas{ \xi_{n-j,m-j}^3R_{u}^{-b,-a,-2-N}(\theta_{n-j}+a+b)  }{u\in U_3}\\
    \dosfilas{\xi_{n-j,m-j}^4R_{u}^{-a,-b,-2-N}(\theta_{n-j}+a+b) }{u\in U_4}
  \end{array}\right|.
\end{equation}
Then

\noindent
(1) The polynomials $q_n$, $0\le n\le N+m_3+m_4$, are orthogonal and have non-null norms with respect to the measure
\begin{align*}
\tilde \rho^{\F,\h}_{a,b,N}&=\prod_{f\in F_1}(b+N+1-f-x)\prod_{f\in F_2}(x+a+1-f)\\&\quad\quad \quad \times
\prod_{f\in F_3}(N+f-x) \prod_{f\in F_4}(x+f_{4,M}+1-f)\rho_{\tilde a,\tilde b ,\tilde N}(x+f_{4,M}+1),
\end{align*}
where
\begin{equation*}\label{abnt}
\tilde a=a -f_{2,M}-f_{4,M}-h_2-1,\quad \tilde b=b -f_{1,M}-f_{3,M}-h_1-h_3,
\quad \tilde N=N+f_{3,M}+f_{4,M}+h_3+1,
\end{equation*}
and $\rho_{a,b,N}$ is the Hahn weight \eqref{pesoH}.

(2) The polynomials $q_n$, $0\le n\le N+m_3+m_4$, are eigenfunctions of a higher-order difference operator of the form (\ref{doho}) with
$$
-s=r=\sum_{f\in F_4}f-\sum_{f\in F_1,F_2,F_3}f-\sum_{i=1}^4\binom{n_{F_i}}{2}+\sum_{i=1}^3n_{F_i}(f_{i,M}+h_i)+1
$$
(which can be constructed using Theorem \ref{Teor1}).
\end{theorem}

\begin{proof}
Notice that the assumption (\ref{appm1}) on the parameters $a$ and $b$ implies that
$$
\tilde a, \tilde b \not =-1,\ldots , -\tilde N, \tilde a + \tilde b \not =-1,\ldots , -2\tilde N -1,
$$
and hence the Hahn weight $\rho_{\tilde a,\tilde b ,\tilde N}(x+f_{4,M}+1)$ is well defined and its support is
$\{-f_{4,M}-1,\ldots ,N+f_{3,M}+h_3\}$. Using the assumptions (\ref{appm1}) on the parameters $a$ and $b$, we deduce that the support of the measure $\tilde \rho^{\F,\h}_{a,b,N}$ is
$$
\{-f_{4,M}-1,\ldots ,N+f_{3,M}+h_3\}\setminus \Big((N+F_3)\cup (-f_{4,M}-1+F_4)\Big) .
$$
Notice that the support is formed by $N+f_{3,M}+f_{4,M}+h_3+2-n_{F_3}-n_{F_4}$ integers. Taking into account  that $U_3=J_{h_3}(F_3)$, $U_4=I(F_4)$ and (\ref{neci}) we get
$$
N+f_{3,M}+f_{4,M}+h_3+2-n_{F_3}-n_{F_4}=N+m_3+m_4+1.
$$
Before going on with the proof we comment on the assumption that $\Omega_{a,b,N}^{\U,\h} (n)\not =0$ for $0\le n\le N+m_3+m_4+1$. If $F_1\not =\emptyset$ or $F_2\not =\emptyset$, since the sequence $\varepsilon _n^h$, $h\in \UU_1\cup\UU_2$, vanish for $n=N+1$, it is not difficult to see that $\Omega_{a,b,N}^{\U,\h} (n) =0$ for $N+m_3+m_4+2\le n\le N+m$. If $F_1=\emptyset$ and $F_2=\emptyset$, the situation is different, and, except for exceptional values of the parameters $a, b$ and $N$, we have $\Omega_{a,b,N}^{\U,\h} (n)\not =0$ for all $n\ge 0$.
In this case, the polynomials $(q_n)_n$ are defined for all $n\ge 0$ and always have degree $n$. However, it is not difficult to see that for $n\ge N+m_3+m_4+1$, the polynomial $q_n(x)$ vanishes in the support of $\tilde \rho^{\F,\h}_{a,b,N}$. Hence they are still orthogonal with respect to this measure but have null norms.
This is completely analogous to the situation with the Hahn polynomials $h_n^{a,b,N}$, which are defined for all $n\ge 0$ (except when $a,a+b+1=-1,-2,\ldots $) and always have degree $n$. But if $n\ge N+1$, they vanish in the support of its weight (see (\ref{mdc1})).

To prove (1) of the Theorem, we use the strategy of the Section \ref{ssi}.

We need some notation. Write $Z^h_j$, $h=1,\ldots ,m, j\ge 0$, for the polynomials
\begin{equation}\label{defz}
Z^h_j(x)=\begin{cases} R_j^{-b,-a,a+b+N}(x+a+b),&h\in \UU_1,\\
R_j^{-a,-b,a+b+N}(x+a+b),&h\in \UU_2,\\
R_{j}^{-b,-a,-2-N}(x+a+b),&h\in \UU_3,\\
R_{j}^{-a,-b,-2-N}(x+a+b),&h\in \UU_4.\end{cases}
\end{equation}
The assumptions  (\ref{appm2}) on the parameters $a$ and $b$ implies that these dual Hahn polynomials are well defined and have degree $j$.
Denote by $G$ and $\tilde G$ the $m$-tuples
\begin{align}\label{defgg}
G&=(u_{1}^{1\rceil},\ldots, u_{m_1}^{1\rceil},u_{1}^{2\rceil},\ldots, u_{m_2}^{2\rceil},u_{1}^{3\rceil},\ldots, u_{m_3}^{3\rceil},u_{1}^{4\rceil},\ldots, u_{m_4}^{4\rceil})=(g_1,\ldots, g_m),\\
\nonumber\tilde G&=(\tilde g_1,\ldots, \tilde g_m),
\end{align}
where
$$
\tilde g_i=\begin{cases} u_{i}^{1\rceil}-b-N,& i\in \UU_1, \\-u_{i}^{2\rceil}+a,& i\in \UU_2,
\\-u_{i}^{3\rceil}-N-1,& i\in \UU_3,\\u_{i}^{4\rceil}+1,& i\in \UU_4. \end{cases}
$$
Finally, write $\p _{\tilde G}$ for the polynomial
$$
\p_{\tilde G}(x)=\prod_{u\in U_1}(x-u+b+N)\prod_{u\in U_2}(x+u-a)\prod_{u\in U_3}(x+u+N+1)\prod_{u\in U_4}(x-u-1).
$$
It is easy to see that $\p_{\tilde G}$ has simple roots if and only if
\begin{align}\nonumber
&u+v-a-b-N,u+w-b+1,v+z-a+1,w+z+N+2\not =0,\\\label{yas}
&z-u+b+N+1, w-v+a+N+1\not =0,
\end{align}
for $u\in U_1,v\in U_2,w\in U_3,z\in U_4$. These constraints follow easily from the assumptions (\ref{appm1})  on the parameters $a$ and $b$. Hence, $\p_{\tilde G}$ has simple roots.

Proceeding as in the proof of Theorem 1.1 in \cite{ddI}, one can prove that
\begin{align*}
c^{\F,\h}_{a,b,N}\langle \tilde \rho^{\F,\h}_{a,b,N},h_n^{a,b,N}\rangle &=(-1)^n\sum_{i=1}^{m}\frac{\xi^i_{n,n+1}Z^i_{g_i}(n)}{\p _{\tilde G}'(\tilde g_i)Z^i_{g_i}(-1)},\quad n\ge 0,\\
\sum_{i=1}^{m}\frac{Z^i_{g_i}(n)}{\xi^i_{-1,-n-1}\p _{\tilde G}'(\tilde g_i)Z^i_{g_i}(-1)}&=0,\quad 1-m\le n<0,\\
\sum_{i=1}^{m}\frac{Z^i_{g_i}(n)}{\xi^i_{-1,m-1}\p _{\tilde G}'(\tilde g_i)Z^i_{g_i}(-1)}&\not =0,
\end{align*}
where $c^{\F,\h}_{a,b,N}$ is the constant independent of $n$ given by
$$
c^{\F,\h}_{a,b,N}=(-1)^{m_2+m_4}\Gamma (a+1)\Gamma(b)(a+b+1)_{N+1}(N+2)_{f_{3,M}+f_{4,M}+h_3},
$$
and  $\xi_{x,y}^i$ are defined by (\ref{xit}).

From the recurrence relation for the Hahn polynomials, we get that $a_{n}c_n\not =0$ for $0\le n\le N$, but $c_{N+1}=0$. It is also easy to check that $\varepsilon_n^h\not =0$, $h=1,\ldots, m$, when $n$ is a negative integer.
Since we assume that $\Omega_{a,b,N}^{\U,\h} (n)\not =0$ for $0\le n\le N+m_3+m_4+1$, the orthogonality of the polynomials $q_n$, $0\le n\le N+m_3+m_4$, with respect to $\tilde \rho_{a,b,N}^{\F,\h}$ is now a consequence of the Lemmas \ref{lemRme}, \ref{lort2} and the identities (\ref{defvardh}) and (\ref{xit}). They have also non-null norms.

We now prove (2) of the Theorem.
Using (\ref{xit}), it is straightforward to see that $\Omega_{a,b,N}^{\U,\h}(x)$ coincides with the (quasi) Casorati determinant
$$
\det \left(\xi_{x-j,m-j}^lY_l(\theta_{x-j})\right)_{l,j=1}^m,
$$
where $Y_l(x)=Z^l_{g_l}(x)$, and $Z^l_j$, $l=1,\ldots, m, j\ge 0$, and $G=\{g_1,\ldots, g_m\}$ are defined by (\ref{defz}) and (\ref{defgg}), respectively.
Consider the particular case of the polynomial $P$ (\ref{def2p}) in Lemma \ref{lgp1}
for $Y_l(x)=Z^l_{g_l}(x)$ (and denote it again by $P$), and write $S$ for the rational function
\begin{equation*}\label{GGH1}
S(x)=\sigma_{x-\frac{m-1}{2}}\frac{(D^{1;m-1}_{x-1})^{m_2+m_4}(D^{2;m-1}_{x-1})^{m_1+m_2}}{p(x)q(x)},
\end{equation*}
where $p$ and $q$ are the polynomials defined by (\ref{def1p}) and (\ref{def1q}), respectively.
A simple computation shows that $S(x)\Omega_{a,b,N}^{\U,\h}(x)=\sigma_{x-\frac{m-1}{2}}P(x)$.

Notice that the rational function $S$ is the particular case of the rational function (\ref{GGH}) in Lemma \ref{l5.2} for $\Xi=1$. Since obviously $\Xi=1$ is a polynomial invariant under the action of $\I^{a+b-m-1}$,
we get, as a direct consequence of Lemma \ref{l5.2} and Theorem \ref{Teor1}, that the polynomials $q_n$, $0\le n\le N+m_3+m_4$, are eigenfunctions of a higher-order difference operator $D_{q,S}$ in the algebra of difference operators $\A$ (\ref{algdiffd}), explicitly given by (\ref{Dq}).

We now compute the order of $D_{q,S}$.

Since $S\Omega _{a, b, N}^{\U,\h}=\sigma_{x-\frac{m-1}{2}}P$, Lemma  \ref{lgp1} gives that the degree of $S\Omega _{a, b, N}^{\U,\h}$ is
$d+1$, where $d$ is the degree of $P$ given by
$$
d=2\sum_{u\in U_1,U_2,U_3,U_4}u-2\sum_{i=1}^4\binom{m_i}{2},
$$
(notice that the assumption (\ref{yas0}) in Lemma \ref{lgp1} is just (\ref{yas}) above).
Hence the polynomial $P_S$ defined by
$$
P_S(\theta_x)-P_S(\theta_{x-1})=S(x)\Omega _{a, b, N}^{\U,\h}(x)+S(x+m)\Omega _{a, b, N}^{\U,\h}(x+m)
$$
has degree $d/2+1$ (since $\theta_x$ is a polynomial of degree $2$), that is,
$$
\sum_{u\in U_1,U_2,U_3,U_4}u-\sum_{i=1}^4\binom{m_i}{2}+1.
$$
Taking into account that the $m$-tuple $G$ (\ref{defgg}) is formed by the sets $J_{h_i}(F_i)$, $i=1,2,3$, and $I(F_4)$, the definitions of the involution $I$ (\ref{dinv}) and the transform $J_h$ (\ref{dinv2}) give
\begin{align*}
&\sum_{u\in U_1,U_2,U_3,U_4}u-\sum_{i=1}^4\binom{m_i}{2}+1\\ &\qquad=
\sum_{f\in F_4}f-\sum_{f\in F_1,F_2,F_3}f-\sum_{i=1}^4\binom{n_{F_i}}{2}+\sum_{i=1}^3n_{F_i}(f_{i,M}+h_i)+1\\&
\qquad =r.
\end{align*}
That is, $P_S$ is a polynomial of degree $r$. Consider the coefficients $B$ and $D$ of $\Sh_{1}$ and $\Sh_{-1}$ in the second-order difference operator $D_{a,b,N}$ for the Hahn polynomials (\ref{dopHahn}). We then deduce that
the operator $\frac{1}{2}P_S(D_{a,b,N})$ has the form
$$
\sum _{l=-r}^{r}\tilde h_{l}(x)\Sh_{l},
$$
where $\tilde h_{r}(x)=u_S\prod_{j=0}^{r-1}B(x+j)$, $\tilde h_{-r}(x)=u_S\prod_{j=0}^{r-1}D(x-j)$ and $u_S$ denotes the leading coefficient of the polynomial $ P_S$.  Using (\ref{dopHahn}), we deduce that both $\tilde h_{-r}$ and $\tilde h_r$ are polynomials of degree $2r$.

Consider now the coefficients $\tilde B_h$ and $\tilde D_h$ of $\Delta$ and $\nabla$ in any of the $\D$-operators $\D_h$ for the Hahn polynomials (see (\ref{dohh1}) and (\ref{dohh2})). Notice that $\tilde B_h=0$, $h\in \UU_1\cup \UU_4$, and $\tilde D_h=0$, $h\in \UU_2\cup \UU_3$. Using Lemmas \ref{lgp2} and \ref{lgp1}, we can conclude that the polynomials $M_h$ (\ref{emeiexp}) have degree at most
$v_h=2r-2g_h-1$. This implies that the polynomials $\tilde M_h$ \eqref{ass1} have degree at most $r-g_h-1$.

Assume now $h\in \UU_1\cup \UU_4$. Since $Y_h$ has degree $g_h$, we get that the operator $\tilde M_h(D_{a,b,N})\D_h Y_h(D_{a,b,N})$ has the form
$$
\sum _{l=-r}^{r-1}\tilde h_{l}(x)\Sh_{l},
$$
where
$$
\tilde h_{-r}(x)=u_Yu_M\tilde D_h(x-v_h)\prod_{j=0;j\not =v_h}^{r-1}D(x-j),
$$
$u_Y$ is the leading coefficient of $Y_h$ and
$u_M$ is the coefficient of $x^{v_h}$ in $\tilde M_h$. As before, we deduce that  $\tilde h_{-r}$ is a polynomial of degree $2r-1$.

On the other hand, if $h\in \UU_2\cup \UU_3$, since $Y_h$ has degree $g_h$, we get that the operator $\tilde M_h(D_{a,b,N})\D_h Y_h(D_{a,b,N})$ has the form
$$
\sum _{l=-r+1}^{r}\tilde h_{l}(x)\Sh_{l},
$$
where
$$
\tilde h_{r}(x)=u_Yu_M\tilde B_h(x-v_h)\prod_{j=0;j\not =v_h}^{r-1}B(x-j).
$$
As before, we deduce that  $\tilde h_{r}$ is a polynomial of degree $2r-1$.

To complete the proof of (2) it is enough to take into account the expression of $D_{q,S}$ given by (\ref{Dq}).

\end{proof}

\begin{corollary}\label{jodme}
Let $\F=(F_1,F_2,F_3,F_4)$  be a quartet of finite sets of positive integers (the empty set is allowed, in which case we take $\max F=-1$). Let $a$ and $b$ be real numbers satisfying
\begin{align*}
&\hspace{2cm} a, b,a+b\not =-1,-2,-3,\ldots,\\
&a+f_{2,M}+f_{4,M}+1\not=0,1,2,\ldots, \quad \mbox{if \quad $F_2$ or $F_4\not =\emptyset$,\quad and} \\
&b+f_{1,M}+f_{3,M}+1\not=0,1,2,\ldots, \quad \mbox{if \quad $F_1$ or $F_3\not =\emptyset$}.
\end{align*}
Consider the weight $\rho _{a,b,N}^{\F}$ defined by
\begin{equation}\label{udspm}
\rho _{a,b,N}^{\F}=\prod_{f\in F_1}(b+N+1+f-x)\prod_{f\in F_2}(x+a+1+f)\prod_{f\in F_3}(N-f-x)\prod_{f\in F_4}(x-f)\rho _{a,b,N},
\end{equation}
where $\rho _{a,b,N}$ is the Hahn weight \eqref{pesoH}. Assume that
$$
\Omega_{a+f_{2,M}+f_{4,M}+2,b+f_{1,M}+f_{3,M}+2,
N-f_{3,M}-f_{4,M}-2}^{\U}(n)\not =0, \quad 0\le n\le N+m_3+m_4+1,
$$
where $\U=(I( F_1),I(F_2),I(F_3),I(F_4))$.
Then the measure $\rho _{a,b,N}^{\F}$ has associated a sequence of orthogonal polynomials and they are eigenfunctions of a higher-order difference operator of the form (\ref{doho}) with
\begin{equation}\label{ordfi}
-s=r=\sum_{f\in F_1,F_2,F_3,F_4}f-\sum_{i=1}^4\binom{n_{F_i}}{2}+1
\end{equation}
(which can be constructed using Theorem \ref{Teor1}).
\end{corollary}

\begin{proof}
If we write $\tilde F_j=\{f_{j,M}-f+1,f\in F_j\}$, $j=1,2,3$,  and $h_j=\begin{cases} \min F_j,& \mbox{if $F_j\not =\emptyset$,}\\
1,& \mbox{if $F_j\not =\emptyset$,}\end{cases}$ using (\ref{dinv}) and (\ref{dinv2}), one straightforwardly has $J_{h_j}(\tilde F_j)=I(F_j)$. Write now $\tilde \F=(\tilde F_1,\tilde F_2,\tilde F_3,F_4)$. It is now easy to see that
$$
\rho _{a,b,N}^{\F}=\tilde \rho_{a+f_{2,M}+f_{4,M}+2,b+f_{1,M}+f_{3,M}+2,
N-f_{3,M}-f_{4,M}-2}^{\tilde \F,\h}(x-f_{4,M}-1).
$$
The corollary is then a straightforward consequence of the previous Theorem.
\end{proof}

\begin{remark}

The hypothesis on $\Omega_{a,b,N}^{\U,\h}(n)\not =0$, for $0\le n\le N+m_3+m_4+1,$ in the previous Theorem and Corollary is then sufficient for the existence of a sequence of orthogonal polynomials with respect to the (possible signed) measure $\tilde \rho_{a,b,N} ^{\F,\h}$. We guess that this hypothesis is also necessary for the existence of such sequence of orthogonal polynomials.

\end{remark}

\begin{remark}

Notice that there are different sets $F_3$ and $F_4$ for which the measures $\rho _{a,b,N}^{\F}$
(\ref{udspm}) are equal. Each of these possibilities provides
a different representation for the orthogonal polynomials with respect to $\rho _{a,b,N}^{\F}$ in the form (\ref{qusch}) and a different higher-order difference operator with respect to which they are eigenfunctions. It is not difficult to see that only one of these possibilities satisfies the condition $f_{3,M},f_{4,M}<N/2$. This is the more interesting choice because it minimizes the order $2r$ of the associated higher-order difference operator.
This fact will be clear with an example. Take $N=100$ and the measure $\mu=(x-1)(x-5)(x-68)\rho _{a,b,N}^{\mathcal F}$. There are eight couples of different sets $F_3$ and $F_4$ for which the measures $\mu$ and $\rho _{a,b,N}^{\F}$ (\ref{udspm}) coincide (except for a sign). They are the following
\begin{align*}
F_3&=\emptyset, F_4=\{1,5,68\}, \quad &F_3&=\{32,95,99\}, F_4=\emptyset,\\
F_3&=\{32\}, F_4=\{1,5\},\quad &F_3&=\{95\}, F_4=\{1,68\},\quad &F_3&=\{99\}, F_4=\{5,68\},\\
F_3&=\{32,95\}, F_4=\{1\},\quad &F_3&=\{32,99\}, F_4=\{5\},\quad &F_3&=\{95,99\}, F_4=\{68\}.
\end{align*}
Only one of these couples satisfies the assumption $f_{3,M},f_{4,M}<N/2$: $F_3=\{32\}$, $F_4=\{1,5\}$. Actually, it is easy to check that
 this couple
minimizes the number
$$
\sum_{f\in F_3,F_4}f-\binom{n_{F_3}}{2}-\binom{n_{F_4}}{2}+1.
$$
Hence, it also minimizes de order $2r$ (see \eqref{ordfi}) of the difference operator with respect to which the polynomials $(q_n)_n$ are eigenfunctions.

\end{remark}

\end{document}